\documentclass{amsproc}%
\usepackage{amsfonts}
\usepackage{amsmath}
\usepackage{amssymb}
\usepackage{graphicx}
\usepackage{hyperref}%
\providecommand{\U}[1]{\protect\rule{.1in}{.1in}}
\theoremstyle{plain}

\newtheorem{corollary}{Corollary}

\newtheorem{definition}{Definition}
\newtheorem{example}{Example}

\newtheorem{lemma}{Lemma}

\newtheorem{proposition}{Proposition}

\newtheorem{theorem}{Theorem}
\numberwithin{equation}{section}
\begin{document}
\title[{\normalsize On classical 1-absorbing prime submodules}]{{\normalsize On classical 1-absorbing prime submodules}}
\author{Zeynep Y\i lmaz}
\address{Department of Mathematics, Yildiz Technical University, Istanbul, Turkey;
Orcid: 0009-0009-8107-504X}
\email{zeynepyilmaz@hotmail.com.tr}
\author{Bayram Ali Ersoy}
\address{Department of Mathematics, Yildiz Technical University, Istanbul, Turkey;
Orcid: 0000-0002-8307-9644}
\email{ersoya@yildiz.edu.tr}
\author{\"{U}nsal Tekir}
\address{Department of Mathematics, Marmara University, Istanbul, Turkey;\\
Orcid: 0000-0003-0739-1449}
\email{utekir@marmara.edu.tr}
\author{Suat Ko\c{c}}
\address{Department of Mathematics, Istanbul Medeniyet University, Istanbul, Turkey;
Orcid: 0000-0003-1622-786X }
\email{suat.koc@medeniyet.edu.tr}
\author{Serkan Onar}
\address{Department of Mathematical Engineering, Yildiz Technical University, Istanbul,
Turkey; Orcid: 0000-0003-3084-7694}
\email{serkan10ar@gmail.com}

\begin{abstract}
In this study, we aim to introduce the concept of classical 1-absorbing prime
submodules of a nonzero unital module $M\ $over a commutative ring $A\ $with
unity. A proper submodule $P$ of $M$ is said to be a classical 1-absorbing
prime submodule, if for each $m\in M$ and nonunits $a,b,c\in A,$ $abcm\in P$
implies that $abm\in P$ or $cm\in P$. We give many examples and properties of
classical 1-absorbing prime submodules. Also, we investiage the classical
1-absorbing prime submodules of tensor product$\ F\otimes M$ of a (faithfully)
flat $A$-module $F$ and any $A$-module $M.\ $Furthermore, we determine
classical prime, classical 1-absorbing prime and classical 2-absorbing
submodules of amalgamated duplication $M\bowtie I$ of an $A$-module $M\ $along
an ideal $I$. Also, we characterize local rings $(A,\mathfrak{m})$ with
$\mathfrak{m}^{2}=0$ in terms of classical 1-absorbing prime submodules.

\end{abstract}
\maketitle

\section{Introduction}

In this paper, we focus only on commutative rings having nonzero identity and
nonzero unital modules. Let $A$ always denote such a ring and $M$ denote such
an $A$-module. A proper submodule $P$ of $M\ $is said to be a \textit{prime
submodule} if whenever $am\in P$ for some $a\in A$ and $m\in M,\ $then
$a\in(P:_{A}M)$ or $m\in P,\ $where $(P:_{A}M)$ is the annihilator
$ann(M/P)\ $of $A$-module $M/P$, that is, $(P:_{A}M)=\{x\in A:xM\subseteq
P\}.\ $In this case, $(P:_{A}M)$ is a prime ideal of $A.\ $However, the
converse is not true in general. For instance, consider $%
%TCIMACRO{\U{2124} }%
%BeginExpansion
\mathbb{Z}
%EndExpansion
$-module $%
%TCIMACRO{\U{2124} }%
%BeginExpansion
\mathbb{Z}
%EndExpansion
^{2}$ and note that $P=0\times2%
%TCIMACRO{\U{2124} }%
%BeginExpansion
\mathbb{Z}
%EndExpansion
$ is not a prime submodule since $2(0,1)\in P$, $2\notin(P:_{%
%TCIMACRO{\U{2124} }%
%BeginExpansion
\mathbb{Z}
%EndExpansion
}%
%TCIMACRO{\U{2124} }%
%BeginExpansion
\mathbb{Z}
%EndExpansion
^{2})$ and $(0,1)\notin P$. But $(P:_{%
%TCIMACRO{\U{2124} }%
%BeginExpansion
\mathbb{Z}
%EndExpansion
}%
%TCIMACRO{\U{2124} }%
%BeginExpansion
\mathbb{Z}
%EndExpansion
^{2})=0$ is a prime ideal of $%
%TCIMACRO{\U{2124} }%
%BeginExpansion
\mathbb{Z}
%EndExpansion
.\ $In 2004, Behboodi and Koohy introduced and extensively studied the notion
of weakly prime submodules (See, \cite{X9}, \cite{X10} and \cite{XX}).
According to \cite{XX}, a proper submodule $P$ of $M$ is said to be a
\textit{weakly prime} \textit{submodule }if whenever $a,b\in A$ and $m\in M$
such that $abm\in P$, then $am\in P$ or $bm\in P$.\ Note that every prime
submodule is also weakly prime but the converse is not true in general. For
instance, the submodule $P=0\times2%
%TCIMACRO{\U{2124} }%
%BeginExpansion
\mathbb{Z}
%EndExpansion
$ of $%
%TCIMACRO{\U{2124} }%
%BeginExpansion
\mathbb{Z}
%EndExpansion
$-module $%
%TCIMACRO{\U{2124} }%
%BeginExpansion
\mathbb{Z}
%EndExpansion
^{2}$ is a weakly prime submodule which is not prime. One can easily see that
a proper submodule $P$ of $M\ $is a weakly prime submodule if and only if
$(P:_{A}m)=\{a\in A:am\in P\}$ is a prime ideal of $A$ for every $m\notin P$.
So far, the concept of weakly prime submodules has drawn the attention of many
authors and has been studied in various papers. For example, Azizi in his
paper \cite[Lemma 3.2 and Theorem 3.3]{X4} studied the weakly prime submodules
of tensor product $F\otimes M$ for a flat (faithfully flat) $A$-module $F$ and
any $A$-module $M$ (See, \cite[Lemma 3.2 and Theorem 3.3]{X4}). Here, we
should mention that after some initial works, many authors studied the same
notion under the name "classical prime submodule", and in this paper we prefer
to use classical prime submodule instead of weakly prime submodule. The notion
of prime ideals/submodules and their generalizations have a distinct place in
commutative algebra since not only they are used in the characterization of
important classes of rings/modules but also they have some applications to
other areas such as General topology, Algebraic geometry, Graph theory and
etc. (See, \cite{Yildiz}, \cite{Yildiz2}, \cite{Koc} and \cite{X5}). One of
the important generalization of prime ideal which is called 2-absorbing ideal
was first introduced by Badawi in 2007. According to \cite{X5} a proper ideal
$I\ $of $A$ is said to be a \textit{2-absorbing ideal} if whenever $abc\in I$
for some $a,b,c\in A,\ $then $ab\in I$ or $ac\in I$ or $bc\in I.\ $In the last
16 years, where as the theory of 2-absorbing ideal has been developed by many
authors, there are still some open questions about 2-absorbing ideals. For
more information about 2-absorbing ideals, the reader may consult
\cite{Badawi2}. Motivated by the notion of 2-absorbing ideals, Mostafanasab et
al. defined classical 2-absorbing submodules as follows: a proper submodule
$P$ of $M\ $is said to be a classical 2-absorbing submodule if $abcm\in P$ for
some $a,b,c\in A$\ and $m\in M$ implies that $abm\in P$ or $acm\in P$ and
$bcm\in P$ \cite{X18}. Note that every classical prime submodule is also
classical 2-absorbing but the converse is not true in general. Consider $%
%TCIMACRO{\U{2124} }%
%BeginExpansion
\mathbb{Z}
%EndExpansion
$-module $%
%TCIMACRO{\U{2124} }%
%BeginExpansion
\mathbb{Z}
%EndExpansion
_{p^{2}}$ and $P=\langle\overline{0}\rangle$ zero submodule, where $p\ $is a
prime number. Then $P\ $is not a classical prime since $p^{2}\overline{1}\in
P$ and $p\overline{1}\notin P.$ On the other hand, one can easily see that $P$
is a classical 2-absorbing submodule. In a recent study, Yassine et al.
introduced the notion of 1-absorbing prime ideal which is an intermediate
class between prime ideals and 2-absorbing ideals: a proper ideal $I$ of
$A\ $is said to be a 1-absorbing prime ideal if whenever $abc\in I$ for some
nonunits $a,b,c$ in $A,\ $then either $ab\in I$ or $c\in I\ $\cite{X8}. The
authors proved that if a ring $A\ $admits a 1-absorbing prime ideal which is
not a prime ideal, then $A\ $is a local ring, that is, it has unique maximal
ideal. By definitions, we have the following implications
\[
\text{prime ideal }\Rightarrow\text{ 1-absorbing prime ideal }\Rightarrow
\text{ 2-absorbing prime ideal}%
\]
However the above arrows are not reversible in general. For instance, in a
local ring $(A,\mathfrak{m})$ where $\mathfrak{m}^{2}\neq\mathfrak{m}$, then
$\mathfrak{m}^{2}$ is a 1-absorbing prime ideal which is not prime. On the
other hand, in the polynomial ring $k[x,y]\ $where $k$ is a field, $P=(xy)$ is
a 2-absorbing ideal which is not 1-absorbing prime. Our aim in this paper is
to introduce and study classical 1-absorbing prime submodules, and also use
them to characterize some important class of rings. A proper submodule $P$ of
$M\ $is said to be a \textit{classical 1-absorbing prime submodule} if
whenever $abcm\in P$ for some nonunits $a,b,c\in A$ and $m\in M,\ $then
$abm\in P$ or $cm\in P.$ For the completeness of the paper, now we shall give
some notions and notations which are needed in the sequel. Let $A$ be a ring
and $M\ $an $A$-module. For every proper ideal $I$ of $A,\ $the radical
$\sqrt{I}$ of $I$ is defined as the intersection of all prime ideals that
contains $I$, or equivalently, $\sqrt{I}=\{a\in A:a^{k}\in I$ for some $k\in%
%TCIMACRO{\U{2115} }%
%BeginExpansion
\mathbb{N}
%EndExpansion
\}\ $\cite{X6}. For every element $x\in A$\ $(m\in M),\ $principal ideal
(cyclic submodule)\ generated by $x\in A\ $($m\in M)$ is denoted by $\langle
x\rangle=xA$ ($\langle m\rangle=Am).\ $Also, for any subset $K$ of $M$ and any
submodule $P$ of $M,\ $the residual of $P\ $by $K$\ is defined by
$(P:_{A}K)=\{a\in A:aK\subseteq P\}.\ $In particular, if $P=\langle0\rangle$
is the zero submodule and $K=\{m\}$ is a singleton for some $m\in M,\ $we
prefer to use $ann(m)\ $instead of $(\langle0\rangle:_{A}K).\ $Similarly, we
prefer to use $ann(M)\ $to denote the annihilator of $M\ $instead of
$(0:_{A}M)$ \cite{X17}. Recall from \cite{X17} that $M\ $is said to be a
\textit{faithful module} if $ann(M)\ $is the zero ideal of $A.\ $An $A$-module
$M$ is called a \textit{multiplication module} if every submodule $P$ of $M$
has the form $IM$ for some ideal $I$ of $A$ \cite{X7}. In this case, we have
$P=(P:_{A}M)M.\ $Also for any submodule $P$ of $M\ $and any subset $J$ of $A$,
the residual of $P$ by $J$ is defined by $(P:_{M}J)=\{m\in M:Jm\subseteq
P\}.\ $Among other results in Section 2, first we show that the arrows in the
following diagram are not reversible in general:
\[
P\text{ is classical prime }\Rightarrow P\ \text{is 1-absorbing classical
prime }\Rightarrow P\text{ is 2-absorbing classical prime}%
\]
(See, Example \ref{ex1} and Example \ref{ex2}).\ We give some
characterizations of classical 1-absorbing prime submodules in general modules
over commutative rings (See, Theorem \ref{tmain} and Theorem \ref{tmain2}%
).\ Also, we investigate the stability of classical 1-absorbing prime
submodules under homomorphisms, in factor modules, in localization of modules,
in cartesian product of modules, in multiplication modules (See, Theorem
\ref{thom}, Corollary \ref{cor1}, Proposition \ref{quotient}, Theorem
\ref{tcar}, Theorem \ref{tcargen} and Theorem \ref{tmult}).\ Furthermore, we
determine the classical 1-absorbing prime submodules of tensor product
$F\otimes M$ for a flat (faithfully flat) $A$-module $F$ and any $A$-module
$M$ (See, Theorem \ref{theo6} and Theorem \ref{ttensor}). We characterize
rings $A\ $over which every proper submodule $P\ $of an $A$-module $M\ $is
classical 1-absorbing submodule. In particular, we show that for a ring
$A\ $and every $A$-module $M,\ $every proper submodule of $M$ is classical
1-absorbing prime submodule if and only if $(A,\mathfrak{m})$ is a local ring
with $\mathfrak{m}^{2}=(0)$ (See, Corollary \ref{cormain}). In commutative
algebra, Krull seperation lemma states that if $S\ $is a multiplicatively
closed set and $I\ $is an ideal of $A\ $disjiont from $S,\ $then the set of
all ideals of $A$ disjoint from $S$ has at least one maximal element. Any such
maximal element is a prime ideal \cite{X17}. By defining classical 1-absorbing
prime m-closed set, we prove Krull seperation lemma for classical 1-absorbing
prime submodules (See, Definition \ref{dmultiplicatively}, Proposition
\ref{pro9} and Theorem \ref{tkrull}). Also, we investigate the number of
minimal classical 1-absorbing prime submodules in a Noetherian module (See,
Theorem \ref{tNoetherian}).

Section 3 is dedicated to the study of classical prime, classical 1-absorbing
prime and classical 2-absorbing submodules of amalgamated duplication of
modules over commutative rings. Let\ $A$ be a ring and $I$ be an ideal of $A$.
\textit{The amalgamated duplication of a ring} $A$ \textit{along an ideal}
$I,$ denoted by $A\bowtie I,$ was firstly introduced and studied by D'anna and
Fontana in \cite{Danna}. The amalgamated duplication $A\bowtie
I=\{(a,a+i):a\in A,i\in I\}$ of a ring $A$ along an ideal $I$ is a special
subring of $A\times A$ with componentwise addition and multiplication. In
fact, $A\bowtie I$ is a commutative subring having the same identity of
$A\times A.\ $In \cite{Danna2}, D'anna, Finocchiaro and Fontana introduced the
more general context of amalgamation of rings. The concept of amalgamation of
rings has an important place in commutative algebra and it has been widely
studied by many well known algebraist. One of the most importance aspects of
the amalgamation of rings is to cover basic constructions in commutative
algebra including classical pullbacks, trivial ring extensions, $A+XB[X]$,
$A+XB[[X]]$ and $D+M$ constructions, CPI extensions (in the sense of Boisen
and Sheldon \cite{Boisen}) (See, \cite[Example 2.5-Example 2.7 and Remark
2.8]{Danna2}). In 2018, Bouba, Mahdou and Temakkante extended the concept of
amalgamated duplication of a ring along an ideal to the context of modules as
follows. Let $M\ $be an $A$-module and $I$ be an ideal of $A.\ $\textit{The
amalgamated duplication of an} $A$-\textit{module} $M$ \textit{along an ideal}
$I$, denoted by $M\bowtie I=\{(m,m+m^{\prime}):m\in M,\ m^{\prime}\in IM\}$ is
an $A\bowtie I$-module with componentwise addition and the following scalar
multiplication:\ $(a,a+i)(m,m+m^{\prime})=(am,(a+i)(m+m^{\prime}))$ for each
$(a,a+i)\in A\bowtie I$ and $(m,m+m^{\prime})\in M\bowtie I\ $\cite{Bou}. Note
that if we consider $M=A$ as an $A$-module, then the amalgamated duplication
$M\bowtie I$ of the $A$-module $M\ $along the ideal $I\ $and the amalgamated
duplication $A\bowtie I$ of the ring $A\ $along the ideal $I\ $coincide. If
$P\ $is a submodule of $M$, then it is easy to check that $P\bowtie
I=\{(m,m+m^{\prime}):m\in P,m^{\prime}\in IM\}$ is an $A\bowtie I$-submodule
of $M\bowtie I.\ $Now, one can naturally ask the classical prime, classical
1-absorbing prime and classical 2-absorbing submodules of $M\bowtie I.\ $In
this section, first we find a useful equality for the residual of $P\bowtie I$
by $(a,a+i)\in A\bowtie I$\ or $(m,m+m^{\prime})\in M\bowtie I$ (See, Lemma
\ref{lemfin}). Then by Lemma \ref{lemfin} and Theorem \ref{tmain}, we
determine the classical prime, classical 1-absorbing prime and classical
2-absorbing submodules of $M\bowtie I$ (See, Theorem \ref{tmainnn}).\ 

\section{Properties of classical 1-absorbing prime submodules}

\begin{definition}
A proper submodule $P$ of an $A$-module $M$ is said to be a classical
1-absorbing prime submodule if whenever $abcm\in P$ for some nonunits
$a,b,c\in A$ and $m\in M,\ $then $abm\in P$ or $cm\in P.$
\end{definition}

In the following we give our first result. Since its proof is elemantary, we
omit the proof.

\begin{proposition}
\label{p1}Let $P$ be a proper submodule of an $A$-module $M.\ $Then $P$ is
classical prime $\Rightarrow$ $P\ $is classical 1-absorbing prime
$\Rightarrow$ $P$ is classical 2-absorbing.
\end{proposition}

The arrows in the previous proposition are not reversible in general. See the
following examples.

\begin{example}
\label{ex1}\textbf{(Classical 1-absorbing prime which is not classical
prime)\ }Let $p$ be a prime number and consider $%
%TCIMACRO{\U{2124} }%
%BeginExpansion
\mathbb{Z}
%EndExpansion
_{p^{2}}$-module$\ M=%
%TCIMACRO{\U{2124} }%
%BeginExpansion
\mathbb{Z}
%EndExpansion
_{p^{2}}[x]$ where $x$ is an indeterminate over $%
%TCIMACRO{\U{2124} }%
%BeginExpansion
\mathbb{Z}
%EndExpansion
_{p^{2}}.\ $Let $P=(\overline{0})$ be zero submodule of $M.\ $Since
$\overline{p}^{2}x=\overline{0}$ and $\overline{p}x\neq\overline{0},\ $it
follows that $P$ is not a classical prime. Let $a,b,c$ be nonunits in $%
%TCIMACRO{\U{2124} }%
%BeginExpansion
\mathbb{Z}
%EndExpansion
_{p^{2}}$ and $f(x)\in M$ such that $abcf(x)=\overline{0}.\ $Since $a,b,c$ are
nonunits, then we have $p$ divides $a,b\ $and $c$ which implies that
$abf(x)=\overline{0}.\ $Thus, $P$ is a classical 1-absorbing prime submodule
of $M.$
\end{example}

\begin{example}
\label{ex2}\textbf{(Classical 2-absorbing which is not classical 1-absorbing
prime)\ }Consider $%
%TCIMACRO{\U{2124} }%
%BeginExpansion
\mathbb{Z}
%EndExpansion
$-module $M=%
%TCIMACRO{\U{2124} }%
%BeginExpansion
\mathbb{Z}
%EndExpansion
_{p}\oplus%
%TCIMACRO{\U{2124} }%
%BeginExpansion
\mathbb{Z}
%EndExpansion
_{q}\oplus%
%TCIMACRO{\U{2124} }%
%BeginExpansion
\mathbb{Z}
%EndExpansion
$ and let $P$ be zero submodule of $M$ where $p,q$ are distinct prime
numbers$.\ $Since $p^{2}q(\overline{1},\overline{1},0)\in P$, $p^{2}%
(\overline{1},\overline{1},0)\notin P$ and $q(\overline{1},\overline
{1},0)\notin P,\ $it follows that $P$ is not a classical 1-absorbing prime
submodule of $M.$\ Also, one can easily see that $P$ is a classical
2-absorbing submodule of $M.$
\end{example}

Here, we should mention that there exists $A$-module $M$ which has no
classical 1-absorbing prime submodule. However, since in finitely generated
modules or in multiplication modules there exists always prime submodules,
these two classes of modules have always 1-absorbing prime submodules. In the
following example, $A$-module $M\ $is not finitely generated and not
multiplication module.

\begin{example}
Let $p$ be a (fixed) prime number and $%
%TCIMACRO{\U{2115} }%
%BeginExpansion
\mathbb{N}
%EndExpansion
_{0}=%
%TCIMACRO{\U{2115} }%
%BeginExpansion
\mathbb{N}
%EndExpansion
\cup\{0\}$.\ Then $E\left(  p\right)  :=\left\{  \alpha\in%
%TCIMACRO{\U{211a} }%
%BeginExpansion
\mathbb{Q}
%EndExpansion
/%
%TCIMACRO{\U{2124} }%
%BeginExpansion
\mathbb{Z}
%EndExpansion
:\alpha=\frac{r}{p^{n}}+%
%TCIMACRO{\U{2124} }%
%BeginExpansion
\mathbb{Z}
%EndExpansion
\text{ \ for some }r\in%
%TCIMACRO{\U{2124} }%
%BeginExpansion
\mathbb{Z}
%EndExpansion
\text{ and }n\in%
%TCIMACRO{\U{2115} }%
%BeginExpansion
\mathbb{N}
%EndExpansion
_{0}\right\}  $ is a nonzero submodule of the $%
%TCIMACRO{\U{2124} }%
%BeginExpansion
\mathbb{Z}
%EndExpansion
$-module $%
%TCIMACRO{\U{211a} }%
%BeginExpansion
\mathbb{Q}
%EndExpansion
/%
%TCIMACRO{\U{2124} }%
%BeginExpansion
\mathbb{Z}
%EndExpansion
$. Also by \cite{X17}, we know that all submodules of $E(p)$ has the following
form for some $t\in%
%TCIMACRO{\U{2115} }%
%BeginExpansion
\mathbb{N}
%EndExpansion
\cup\{0\}:$ \ \ \ \ \ \ \ \ \ \ \ \ \ \ \
\[
G_{t}:=\left\{  \alpha\in%
%TCIMACRO{\U{211a} }%
%BeginExpansion
\mathbb{Q}
%EndExpansion
/%
%TCIMACRO{\U{2124} }%
%BeginExpansion
\mathbb{Z}
%EndExpansion
:\alpha=\frac{r}{p^{t}}+%
%TCIMACRO{\U{2124} }%
%BeginExpansion
\mathbb{Z}
%EndExpansion
\text{ for some }r\in%
%TCIMACRO{\U{2124} }%
%BeginExpansion
\mathbb{Z}
%EndExpansion
\right\}  .
\]
However, no $G_{t}$ is a classical 1-absorbing prime submodule of $E\left(
p\right)  $. Indeed, $\frac{1}{p^{t+3}}+%
%TCIMACRO{\U{2124} }%
%BeginExpansion
\mathbb{Z}
%EndExpansion
\in E\left(  p\right)  $ and $p^{3}\left(  \frac{1}{p^{t+3}}+%
%TCIMACRO{\U{2124} }%
%BeginExpansion
\mathbb{Z}
%EndExpansion
\right)  =\frac{1}{p^{t}}+%
%TCIMACRO{\U{2124} }%
%BeginExpansion
\mathbb{Z}
%EndExpansion
\in G_{t}$ but $p^{2}\left(  \frac{1}{p^{t+3}}+%
%TCIMACRO{\U{2124} }%
%BeginExpansion
\mathbb{Z}
%EndExpansion
\right)  =\frac{1}{p^{t+1}}+%
%TCIMACRO{\U{2124} }%
%BeginExpansion
\mathbb{Z}
%EndExpansion
\notin G_{t}$ and $p\left(  \frac{1}{p^{t+3}}+%
%TCIMACRO{\U{2124} }%
%BeginExpansion
\mathbb{Z}
%EndExpansion
\right)  =\frac{1}{p^{t+2}}+%
%TCIMACRO{\U{2124} }%
%BeginExpansion
\mathbb{Z}
%EndExpansion
\notin G_{t}$.
\end{example}

Recall from \cite{X11} that a proper submodule $P$ of $M\ $is said to be a
\textit{1-absorbing prime submodule} if $abcm\in P$ for some nonunits
$a,b,c\in R$ and $m\in M,\ $we have either $ab\in(P:_{A}M)$ or $cm\in P$.

\begin{proposition}
\label{pro2} Let $P$ be a proper submodule of an $A$-module $M$. If $P$ is a
1-absorbing prime submodule of $M$, then $P$ is a classical 1-absorbing prime
submodule of $M$.
\end{proposition}

\begin{proof}
Assume that $P$ is a 1-absorbing prime submodule of $M$. Let $a,b,c\in A$ be
nonunits and $m\in M$ such that $abcm\in$ $P$. Since $P$ is a 1-absorbing
prime submodule of $M$, we have either $ab\in\left(  P:_{A}M\right)  $ or
$cm\in P$.\ The the first case gives that $abm\in P$. The second case leads us
to the claim. Consequently, $P$ is a classical 1-absorbing prime submodule.
\end{proof}

The converse of previous proposition need not be true. See the following example.

\begin{example}
\textbf{(Classical 1-absorbing prime which is not 1-absorbing prime) }Let $p$
be a prime number and consider $%
%TCIMACRO{\U{2124} }%
%BeginExpansion
\mathbb{Z}
%EndExpansion
$-module $M=%
%TCIMACRO{\U{2124} }%
%BeginExpansion
\mathbb{Z}
%EndExpansion
_{p}\oplus%
%TCIMACRO{\U{211a} }%
%BeginExpansion
\mathbb{Q}
%EndExpansion
$.\ Assume that $P\ $is zero submodule of $M.\ $Let $a,b,c$ be nonunits in $%
%TCIMACRO{\U{2124} }%
%BeginExpansion
\mathbb{Z}
%EndExpansion
$ and $(\overline{x},y)\in M$ such that $abc(\overline{x},y)=(0,0).\ $Without
loss of generality we may assume that $y=0$. Then we have $p\ $divides
$a,b,c\ $or $x$. Which implies that $ab(\overline{x},y)=(0,0)$ or
$c(\overline{x},y)=(0,0).\ $Now, choose a prime number $q\neq p.\ $Since
$p^{2}q(\overline{1},0)=(\overline{0},0),\ p^{2}\notin(P:_{%
%TCIMACRO{\U{2124} }%
%BeginExpansion
\mathbb{Z}
%EndExpansion
}M)=0\ $and $q(\overline{1},0)\neq(\overline{0},0),\ $it follows that $P\ $is
not a 1-absorbing prime submodule of $M.$
\end{example}

\begin{proposition}
Let $P$ be a proper submodule of an $A$-module $M$. Then $P$ is a classical
prime submodule of $M$ if and only if $P$ is a classical 1-absorbing prime and
semiprime submodule of $M$.
\end{proposition}

\begin{proof}
The "if part" follows from Proposition \ref{p1}, $P$ is a classical
1-absorbing prime submodule.Now, we want show $P$ is a semiprime submodule.
Let $a\in A$, $m\in M$ such that $a^{2}m\in P$. Since $P$ is a classical prime
submodule we get $am\in M$ and we conclude that $P$ is a semiprime submodule
of $M$ \cite[Proposition 2.1]{XX}. For the "only if part", assume that $P$ is
a classical 1-absorbing prime and semiprime submodule of $M$. Let $abm\in
P\ $for some $a,b\in A\ $and $m\in M.$ If $a$ or $b$ is unit, we have $am\in
P$ or $bm\in P.\ $So assume that $a,b$ are nonunits in $A.$Hence $\left(
a+b\right)  $ be nonunit in $A$.Thus we may write $a\left(  a+b\right)  bm\in
P$. Since $P$ is a classical 1-absorbing prime submodule of $M,$ we have
either $a\left(  a+b\right)  m\in P$ or $bm\in P.\ $If $a\left(  a+b\right)
\in P$ and then $a^{2}m+abm\in P.$Therefore, we get $a^{2}m\in P$. Since $P$
is a semiprime submodule of $M$ we conclude that $am\in P.\ $The second case
leads us to the claim.Thus, $P$ is a classical prime submodule of $M.$
\end{proof}

\begin{theorem}
\label{thom}Let $f:M_{1}\rightarrow M_{2}$ be an $A$-homomorphism.

(i)\ If $P_{2}$ is a classical 1-absorbing prime submodule of $M_{2}$, then
either $f^{-1}(P_{2})=M_{1}$ or$\ f^{-1}(P_{2})$ is a classical 1-absorbing
prime submodule of $M_{1}.$

(ii) If $f$ is surjective and $P_{1}\ $is a classical 1-absorbing prime
submodule of $M_{1}$ containing $Ker(f),\ $then $f(P_{1})$ is a classical
1-absorbing prime submodule of $M_{1}.$
\end{theorem}

\begin{proof}
$(i):\ $Let $P_{2}$ be a classical 1-absorbing prime submodule of $M_{2}$ such
that $f^{-1}(P_{2})\neq M_{1}.\ $Choose nonunits $a,b,c\in A$ and $m\in M_{1}$
such that $abcm\in f^{-1}(P_{2}).\ $Then we have $f(abcm)=abcf(m)\in P_{2}%
.\ $As $P_{2}$ is a classical 1-absorbing prime submodule of $M_{2},$ we
conclude that $abf(m)=f(abm)\in P_{2}$ or $cf(m)=f(cm)\in P_{2}$ which implies
that $abm\in f^{-1}(P_{2})$ or $cm\in f^{-1}(P_{2}).\ $Thus, $f^{-1}(P_{2})$
is a classical 1-absorbing prime submodule of $M_{1}.$

$(ii):\ $Let $abcm^{\prime}\in f(P_{1})$ for some nonunits $a,b,c\in A$ and
$m^{\prime}\in M_{2}.\ $Since $f$ is surjective, there exists $m\in M_{1}$
such that $m^{\prime}=f(m).\ $Then we have $abcm^{\prime}=f(abcm)\in
f(P_{1}).\ $As $P_{1}$ contains $Ker(f),\ $we get $abcm\in P_{1}.\ $Since
$abm\in P_{1}$ or $cm\in P_{1}.\ $Then we conclude that either $abm^{\prime
}=f(abm)\in f(P_{1})$ or $cm^{\prime}=f(cm)\in f(P_{1}).\ $Hence, $f(P_{1})$
is a classical 1-absorbing prime submodule of $M_{1}.$
\end{proof}

\begin{corollary}
\label{cor1}Let $M$ be an $A$-module and $L\subseteq P$ be two submodules of
$M$. Then $P$ is a classical 1-absorbing prime submodule of $M$ if and only if
$P/L$ is a classical 1-absorbing prime submodule of $M/L$.
\end{corollary}

\begin{proof}
Consider the surjective homomorphism $\pi:M\rightarrow M/L$ defined by
$\pi(m)=m+L$ for each $m\in M.\ $The rest follows from Theorem \ref{thom}.
\end{proof}

\begin{theorem}
\label{tmain}Let $M$ be an $A$-module and $P$ a proper submodule of $M.\ $The
following statements are equivalent.

(i)\ $P\ $is a classical 1-absorbing prime submodule.

(ii)\ For every nonunits $a,b,c\in A,\ (P:_{M}abc)=(P:_{M}ab)\cup(P:_{M}c).$

(iii) For every nonunits $a,b\in A$ and $m\in M$ with $abm\notin
P,\ (P:_{A}abm)=(P:_{A}m).$

(iv) For every nonunits $a,b\in A$,\ every proper ideal $I\ $of $A\ $and $m\in
M$ with $aIbm\subseteq P,\ $either $aIm\subseteq P\ $or $bm\in P.$

(v) For every nonunit $a\in A,\ $every proper ideal $I$ of $A$ and $m\in M$
with $aIm\nsubseteq P,\ (P:_{A}aIm)=(P:_{A}m).$

(vi)$\ \ $For every nonunit $a\in A,\ $every proper ideals $I,J\ $of $A\ $and
$m\in M\ $with $aIJm\subseteq P,\ $either $aIm\subseteq P$ or $Jm\subseteq P.$

(vii)\ For every proper ideals $I,J\ $of $A$ and $m\in M$ with $IJm\nsubseteq
P,\ (P:_{A}IJm)=(P:_{A}m).$

(viii) For every proper ideals $I,J,K\ $of $A$ and $m\in M$ with
$IJKm\subseteq P,$\ etiher $IJm\subseteq P$ or $Km\subseteq P.$

(ix) For every $m\in M-P,\ (P:_{A}m)$ is a 1-absorbing prime ideal of $A.$
\end{theorem}

\begin{proof}
$(i)\Rightarrow(ii):$ Suppose that $P$ is a classical 1-absorbing prime
submodule and choose nonunits $a,b,c\in A.$ Let $m\in$ $\left(  P:_{M}%
abc\right)  $. Then we have $abcm\in P\ $which implies that $abm\in P$ or
$cm\in P$.\ Thus, we conclude that $m\in\left(  P:_{M}ab\right)  \cup\left(
P:_{M}c\right)  ,$ that is, $(P:_{M}abc)\subseteq\left(  P:_{M}ab\right)
\cup\left(  P:_{M}c\right)  $. Since the reverse inclusion is always true, we
have the equality $\left(  P:_{M}abc\right)  =\left(  P:_{M}ab\right)
\cup\left(  P:_{M}c\right)  $.

$(ii)\Rightarrow(iii):\ $Let $abm\notin P$ for some nonunits $a,b\in A$ and
$m\in M$. Choose $c\in\left(  P:_{A}abm\right)  $. Then we have $abcm\in P,$
so by (ii) we conclude that $m\in\left(  P:_{M}abc\right)  =\left(
P:_{M}ab\right)  \cup\left(  P:_{M}c\right)  $. Since $abm\notin P,$ we have
$cm\in P$\ which implies that $c\in\left(  P:_{A}m\right)  $.\ Then we get
$\left(  P:_{A}abm\right)  \subseteq\left(  P:_{A}m\right)  .\ $Since the
reverse inclusion is always true, we have the equality $(P:_{A}abm)=(P:_{A}%
m).$

$(iii)\Rightarrow(iv):\ $Let $a,b\in A$ be two nonunits, $I$ be a proper ideal
of $A$ and $m\in M$ with $abIm\subseteq P$. Suppose that $aIm\nsubseteq
P.\ $Then we can find $x\in I$ such that $axm\notin P.\ $Since $abxm\in P,$ by
(iii), we have $m\in(P:_{M}axb)=(P:_{M}ax)\cup(P:_{M}b).\ $On the other hand,
as $axm\notin P,\ $we get $m\in(P:_{M}b)\ $which implies that $bm\in P.$

$(iv)\Rightarrow(v):\ $Let $a\in A$ be a nonunit, $I$ be a proper ideal of $A$
and $m\in M$ with $aIm\nsubseteq P$. Choose $b\in\left(  P:_{A}aIm\right)  $.
Then we have $abIm\subseteq P.$\ Hence by part (iv), we have $aIm\subseteq P$
or $bm\subseteq P$. Since $aIm\nsubseteq P$, we get $bm\in P$, that is,
$b\in\left(  P:_{A}m\right)  $. Consequently, we have $\left(  P:_{A}%
aIm\right)  =\left(  P:_{A}m\right)  $.

$(v)\Rightarrow(vi):\ $Let $a\in A$ be a nonunit, $I,J$ be two proper ideals
of $A$ and $m\in M.\ $Suppose that $aIJm\subseteq P$ and $aIm\nsubseteq P$.
Since $aIJm\subseteq P$ we have $J\subseteq$ $\left(  P:_{M}aIm\right)  $.
Then by part (v), we conclude that $J\subseteq\left(  P:_{A}m\right)  $ which
implies that $Jm\subseteq N$.

$(vi)\Rightarrow(vii):\ $Let $I,J$ be two proper ideals of $A\ $and $m\in M$
with $IJm\nsubseteq P.\ $Choose $a\in(P:_{A}IJm).\ $Then we have
$aIJm\subseteq P.\ $Since $IJm\nsubseteq P,\ $then there exists $x\in J$ such
that $Ixm\nsubseteq P.\ $Now, put $J^{\ast}=aA$ and note that $xIJ^{\ast
}m\subseteq aIJm\subseteq P.\ $Since $xIm\nsubseteq P,\ $again by part
(vi),\ we have $J^{\ast}m\subseteq P$ which implies that $am\in P,\ $that is,
$a\in(P:_{A}m).\ $Hence, we have $(P:_{A}IJm)=(P:_{A}m).$

$(vii)\Rightarrow(viii):\ $Let $I,J,K\ $be proper ideals of $A\ $and $m\in M$
with $IJKm\subseteq P.$\ Then by part (vii), we have $K\subseteq
(P:_{A}IJm)=(P:_{A}m)$ which implies that $Km\subseteq P.\ $

$(viii)\Rightarrow(ix):\ $Let $m\in M-P\ $and $a,b,c\in A$ be nonunits such
that $abc\in(P:_{A}m).\ $Then we have $abcm\in P.\ $Now, put $I=aA,\ J=bA\ $%
and $K=cA.\ $Then note that $IJKm\subseteq P.\ $By part (viii), we conclude
that $abm\in IJm\subseteq P$ or $cm\in Km\subseteq P.\ $Hence, we have
$ab\in(P:_{A}m)$ or $c\in(P:_{A}m).\ $Thus, $(P:_{A}m)$ is a 1-absorbing prime
ideal of $A.$

$(ix)\Rightarrow(i):\ $Let $abcm\in P$ for some nonunits $a,b,c\in A$ and
$m\in M.\ $Assume that $m\notin M.\ $Then by (ix), $(P:_{A}m)$ is a
1-absorbing prime ideal. Since $abc\in(P:_{A}m),\ $we conclude that
$ab\in(P:_{A}m)$ or $c\in(P:_{A}m).\ $Then we have $abm\in P$ or $cm\in
P\ $which completes the proof.
\end{proof}

\begin{corollary}
\label{cormain}Let $A$ be a ring and$\ I$ be a proper ideal of $A$. The
following statements are satisfied.

(i)$\ I$ is a classical 1-absorbing submodule of $A$-module $A$ if and only
if$\ I$ is a 1-absorbing prime ideal of $A$.

(ii)\ For every $A$-module $M$ and every proper submodule $P$ of $M,\ P\ $is a
classical 1-absorbing prime if and only if $A$ is a local ring with unique
maximal ideal $\mathfrak{m}$ of $A$ such that $\mathfrak{m}^{2}=(0).$
\end{corollary}

\begin{proof}
$(i):\ $Suppose that $I\ $is a classical 1-absorbing prime submodule of
$A$-module $A.\ $Then by Theorem \ref{tmain} (ix),$\ \left(  I:_{A}1\right)
=I$ is a 1-absorbing prime ideal of $A$.$\ $For the converse assume that
$I\ $is a 1-absorbing prime ideal of $A.\ $Let $abcx=ab(cx)\in I$ for some
nonunits $a,b,c\in A$ and $x\in A.\ $Since $I\ $is a 1-absorbing prime ideal,
we have $ab\in I$ or $cx\in I$ which implies that $abx\in I$ or $cx\in I.$

$(ii):\ $Assume that, for every $A$-module $M,\ $every proper submodule $P$ of
$M$ is a classical 1-absorbing prime submodule. Put $M=A\ $and apply (i). Then
every proper ideal of $A\ $is a 1-absorbing prime ideal. By \cite[Proposition
4.5]{Abdel}, $A$ is a local ring with unique maximal ideal $\mathfrak{m}$ of
$A$ such that $\mathfrak{m}^{2}=(0).$\ Now, assume that $A$ is a local ring
with unique maximal ideal $\mathfrak{m}$ of $A$ such that $\mathfrak{m}%
^{2}=(0).\ $Let $M$ be an $A$-module and $P$ a proper submodule of $M.\ $Let
$abcm\in P$ for some nonunits $a,b,c\in A$ and $m\in M.\ $Since $a,b\in
\mathfrak{m}$ and $\mathfrak{m}^{2}=(0),\ $we have $abm=0\in P.\ $Thus,
$P\ $is a classical 1-absorbing prime submodule of $M.\ $
\end{proof}

\begin{proposition}
\label{pro4} Let $M$ be an $A$-module and $\left\{  P_{i}:i\in\Delta\right\}
$ be a descending chain of classical 1-absorbing prime submodule of $M$.\ Then
$%
%TCIMACRO{\tbigcap \limits_{i\in\Delta}}%
%BeginExpansion
{\textstyle\bigcap\limits_{i\in\Delta}}
%EndExpansion
K_{i}$\ is a classical 1-absorbing prime submodule of $M$.
\end{proposition}

\begin{proof}
Let $abcm\in$ $%
%TCIMACRO{\tbigcap \limits_{i\in\Delta}}%
%BeginExpansion
{\textstyle\bigcap\limits_{i\in\Delta}}
%EndExpansion
K_{i}$ for some nonunits $a,b,c\in A\ $and $m\in M$. Then we have $abcm\in
K_{i}\ $for each $i\in\Delta.\ $Assume that $abm\notin%
%TCIMACRO{\tbigcap \limits_{i\in\Delta}}%
%BeginExpansion
{\textstyle\bigcap\limits_{i\in\Delta}}
%EndExpansion
K_{i}$. Then there exists $t\in\Delta$ such that $abm\notin K_{t}.\ $Since
$abcm\in K_{t}\ $and $K_{t}\ $is a classical 1-absorbing prime submodule, we
have $cm\in K_{t}.\ $Now choose $i\in\Delta.\ $Then either $K_{t}\subseteq
K_{i}\ $or $K_{i}\subseteq K_{t}.\ $If $K_{t}\subseteq K_{i},$ then $cm\in
K_{i}.\ $If $K_{i}\subseteq K_{t},\ $then $abm\notin K_{i}\ $and $abcm\in
K_{i}.\ $Since $K_{i}\ $is a classical 1-absorbing prime submodule of $M,\ $we
have $cm\in K_{i}.\ $Then, we conclude that $cm\in%
%TCIMACRO{\tbigcap \limits_{i\in\Delta}}%
%BeginExpansion
{\textstyle\bigcap\limits_{i\in\Delta}}
%EndExpansion
K_{i}$ and so $%
%TCIMACRO{\tbigcap \limits_{i\in\Delta}}%
%BeginExpansion
{\textstyle\bigcap\limits_{i\in\Delta}}
%EndExpansion
K_{i}$\ is a classical 1-absorbing prime submodule of $M$.
\end{proof}

A classical 1-absorbing prime submodule $P\ $of $M$ is said to be a
\textit{minimal classical 1-absorbing prime submodule} if for every classical
1-absorbing prime submodule $K$ of $M$ such that $K\subseteq P$, then $K=P$.
Let $L$ be a classical 1-absorbing prime submodule of $M$. Now, set
$\Gamma=\left\{  K:K\text{ is a classical 1-absorbing prime submodule of
}M\text{ and }K\subseteq L\right\}  $. If $\left\{  K_{i}:i\in\Delta\right\}
$ is any chain in $\Gamma$, then $%
%TCIMACRO{\tbigcap \limits_{i\in\Delta}}%
%BeginExpansion
{\textstyle\bigcap\limits_{i\in\Delta}}
%EndExpansion
K_{i}\ $is in $\Gamma$ by Proposition \ref{pro4}. By Zorn's Lemma, there
exists a minimal member of $\Gamma\ $which is clearly a minimal classical
1-absorbing submodule of $M$. If $M$ \ is finitely generated or multiplication
module, then it is clear that$\ M$ contains a minimal classical 1-absorbing
prime submodule.

Recall from \cite{X17} that an $A$-module $M\ $is said to be a Noetherian if
its every submodule is finitely generated. Now, we will show that in a
Noetherian module, there exist finite number of minimal classical 1-absorbing
prime submodules.

\begin{theorem}
\label{tNoetherian}Let $M$ be a Noetherian $A$-module. Then $M$ contains a
finite number of minimal classical 1-absorbing prime submodules.
\end{theorem}

\begin{proof}
Suppose that the claim is not true. Let $\Gamma$ denote the collection of all
proper submodules $P$ of $M$ such that $M/P$ has an infinite number of minimal
classical 1-absorbing submodules. Since $0\in\Gamma,$ $\Gamma$ is nonempty.
Since $M$ is a Noetherian module, there exists a maximal element $Q$ of
$M.\ $If $Q$ is a classical 1-absorbing prime submodule, then by Corollary
\ref{cor1}, $Q/Q=(\overline{0})$ is the unique minimal classical 1-absorbing
prime submodule of $M/Q.\ $Then we have $Q\notin\Gamma,\ $a contradiction.
Thus $Q\ $is not a classical 1-absorbing prime submodule of $M.\ $Then by
Theorem \ref{tmain} (vii), there exist proper ideals $I,J,K$ of $A$ and $m\in
M$ such that $IJKm\subseteq Q$,\ $IJm\nsubseteq Q$ and $Km\nsubseteq
Q.\ $Since $Q\ $is the maximal element of $\Gamma,\ Q+IJm$ and $Q+Km$ are not
in $\Gamma$. This implies that $M/Q+IJm\ $and $M/Q+Km\ $have finite minimal
classical 1-absorbing prime submodules. Let $P^{\ast}/Q$ is a minimal
classical 1-absorbing submodule of $M/Q.\ $Then $IJKm\subseteq Q\subseteq
P^{\ast}$ and $P^{\ast}\ $is a classical 1-absorbing prime submodule, we have
either $IJm\subseteq P^{\ast}$ or $Km\subseteq P^{\ast}.\ $Then we conclude
that $P^{\ast}/(Q+IJm)$ is a minimal classical 1-absorbing prime submodule of
$M/Q+IJm$ or $P^{\ast}/(Q+Km)$ is a minimal classical 1-absorbing prime
submodule of $M/Q+Km.\ $This shows that there are finite number of choices for
$P^{\ast}/Q$. Thus, $M/Q$ has finite number of minimal classical 1-absorbing
prime submodules. This is again a contradiction. Hence $\Gamma=\emptyset$ and
$M$ contains a finite number of minimal classical 1-absorbing prime submodules.
\end{proof}

Recall that an $A$-module $M$ is a multiplication module if every submodule
$N$ of $M$ has the form $IM$ for some ideal $I$ of $A$. Let $N$ and $K$ be
submodules of a multiplication $A$-module $M$ with $N=I_{1}M$ and $K=I_{2}M$
for some ideals $I_{1}\ $and $I_{2}$ of $A$. Ameri in his paper \cite{X1},
defined the product of submodules of multiplication modules as follows: the
product of $N$ and $K$ denoted by $NK$ is defined by $NK:=I_{1}I_{2}M$. The
author showed that the product of $N$ and $K$ is independent of the
presentations $N$ and $K$\ (See \cite[Theorem 3.4]{X1}).

\begin{proposition}
Let $M$ be a multiplication $A$-module and $P$ be a proper submodule of $M$.
The following statements are equivalent.

(i) $P\ $is a classical 1-absorbing prime submodule of $M.$

(ii)\ If $KLNm\subseteq P$ for some proper submodules $K,L,N\ $of $M$ and
$m\in M$,\ then either $KLm\subseteq P$ or $Nm\subseteq P$.
\end{proposition}

\begin{proof}
$(i)\Rightarrow(ii):\ $Suppose that $P\ $is a classical 1-absorbing prime
submodule of $M$ and $KLNm\subseteq P$ for some proper submodules $K,L,N\ $of
$M$ and $m\in M.\ $Since $M\ $is multiplication, we can write
$K=IM,\ L=JM,\ N=QM$ for some proper ideals $I,J,Q\ $of $A.\ $Then note that
$IJQm\subseteq P$.\ As $P\ $is a classical 1-absorbing prime submodule, we
have either $IJm\subseteq P$ or $Qm\subseteq P.\ $Which implies that
$KLm\subseteq P$ or $Nm\subseteq P.\ $

$(ii)\Rightarrow(i):\ $Suppose that $IJQm\subseteq P$ for some proper ideals
$I,J,Q\ $of $A\ $and $m\in M.\ $Now, put $K=IM,\ L=JM,\ N=QM\ $and note that
$KLNm\subseteq P.\ $Thus we have $KLm\subseteq P$ or $Nm\subseteq P$. If
$KLm\subseteq P,\ $we have $IJ(Rm:M)M=IJ(Am)\subseteq P$ which implies that
$IJm\subseteq P.\ $If $Nm\subseteq P,\ $then similarly we have $Qm\subseteq
P.\ $Then by Theorem \ref{tmain}, $P\ $is a classical 1-absorbing prime
submodule of $M.$
\end{proof}

\begin{theorem}
\label{tmain2}Let $M$ be an $A$-module and $P$ be a proper submodule of
$M.$The following statements are equivalent.

(i)\ $P\ $is a classical 1-absorbing prime submodule of $M.\ $

(ii)\ For every nonunits $a,b,c\in A$, $(P:_{M}abc)=(P:_{M}ab)$ or
$(P:_{M}abc)=(P:_{M}c).$

(iii) For every nonunits $a,b,c\in A\ $and every submodule $L$ of
$M,\ abcL\subseteq P$ implies that $abL\subseteq P$ or $cL\subseteq P.$

(iv) For every nonunits $a,b\in A$ and every submodule $L$ of $M$ with
$abL\nsubseteq P,\ (P:_{A}abL)=(P:_{A}L).$

(v) For every nonunits $a,b\in A,$ every proper ideal $I$ of $A$ and every
submodule $L$ of $M$,\ $abIL\subseteq P$ implies that $abL\subseteq P$ or
$IL\subseteq P.$

(vi) For every nonunit $a\in A,\ $every proper ideal $I$ of $A\ $and every
submodule $L$ of $M\ $with $aIL\nsubseteq P,\ (P:_{A}aIL)=(P:_{A}L).$

(vii) For every nonunit $a\in A,\ $every proper ideals $I,J$ of $A$ and every
submodule $L$ of $M,\ aIJL\subseteq P$ implies either $aIL\subseteq P$ or
$JL\subseteq P.$

(viii) For every proper ideals $I,J\ $of $A$ and every submodule $L$ of
$M\ $with $IJL\nsubseteq P,\ (P:_{A}IJL)=(P:_{A}L).$

(ix) For every proper ideals $I,J,K$ of $A$ and every submodule $L$ of
$M,\ IJKL\subseteq P$ implies $IJL\subseteq P$ or $KL\subseteq P.\ $

(x)\ For every submodule $L$ of $M\ $which is not contained in $P,\ (P:_{A}L)$
is a 1-absorbing prime ideal of $A.$
\end{theorem}

\begin{proof}
$(i)\Rightarrow(ii):\ $Follows from Theorem \ref{tmain}\ $(i)\Rightarrow(ii).$

$(ii)\Rightarrow(iii):\ $Since $abcL\subseteq P,\ $we have $L\subseteq
(P:_{M}abc).\ $The rest follows from (ii).

$(iii)\Rightarrow(iv):\ $Let $abL\nsubseteq P$ for some nonunits nonunits
$a,b\in A$ and some submodule $L$ of $M.\ $Let $c\in(P:_{A}abL).\ $Then $c$ is
nonunit and $abcL\subseteq P.\ $Then by $(iii),\ $we have $cL\subseteq P$
which implies that $c\in(P:_{A}L).\ $Hence, we obtain$\ (P:_{A}abL)=(P:_{A}L)$

$(iv)\Rightarrow(v):\ $Suppose that $abIL\subseteq P$ for some nonunits
$a,b\in A$, some proper ideal$\ I$ of $A\ $and some submodule $L$ of $M.\ $We
may assume that $abL\nsubseteq P.\ $Then by part (iv), we have $I\subseteq
(P:_{A}abL)=(P:_{A}L)$, that is, $IL\subseteq P.\ $

$(v)\Rightarrow(vi):\ $Let $a\in A$ be a nonunit, $I\ $be a proper ideal of
$A\ $and $L$ be a submodule of $M\ $with $aIL\nsubseteq P.\ $Then there exists
$x\in I$ such that $axL\nsubseteq P.\ $Let $b\in(P:_{A}aIL).\ $Then
$abIL\subseteq P.\ $Now, put $I^{\ast}=bA.\ $Then $I^{\ast}\ $is a proper
ideal and $axI^{\ast}L\subseteq abIL\subseteq P.\ $Since $axL\nsubseteq
P,\ $by (v), we have $bL\subseteq I^{\ast}L\subseteq P\ $which implies that
$b\in(P:_{A}L).$

$(vi)\Rightarrow(vii):\ $Let $aIJL\subseteq P$ for some nonunit $a\in
A,\ $some proper ideals $I,J\ $of $A\ $and some submodule $L$ of $M.\ $Assume
that $aIL\nsubseteq P.\ $Then by part (vi), we have $J\subseteq(P:_{A}%
aIL)=(P:_{A}L)$ which implies that $JL\subseteq P.\ $

$(vii)\Rightarrow(viii):\ $Let $I,J\ $be two proper ideals of $A\ $and $L$ be
a submodule of $M$ with $IJL\nsubseteq P.\ $Then there exists $x\in J$\ such
that $xIL\nsubseteq P.\ $Let $a\in(P:_{A}IJL).\ $Then we have $aIJL\subseteq
P.\ $Now, put $J^{\ast}=aA\ $and note that $xIJ^{\ast}L\subseteq aIJL\subseteq
P.\ $Since $xIL\nsubseteq P,\ $by part (vii),\ we have $aL\subseteq J^{\ast
}L\subseteq P$ which implies that $a\in(P:_{A}L).$

$(viii)\Rightarrow(ix):\ $Let $IJKL\subseteq P$ for some proper ideals $I,J,K$
of $A$ and some submodule $L$ of $M.$\ If $IJL\subseteq P,\ $then there is
nothing to do. So assume that $IJL\nsubseteq P.\ $Since $IJKL\subseteq P,\ $by
part (viii), we have $K\subseteq(P:_{A}IJL)=(P:_{A}L)\ $which implies that
$KL\subseteq P.\ $

$(ix)\Rightarrow(x):\ $Assume that $abc\in(P:_{A}L)$ for some nonunits
$a,b,c\in A.\ $Now, put $I=aA,\ J=bA\ $and $K=cA.\ $Then note that
$IJKL\subseteq P.\ $Then by part (ix), we have $abL\subseteq IJL\subseteq P$
or $cL\subseteq KL\subseteq P.\ $Then we have $ab\in(P:_{A}L)$ or $c\in
(P:_{A}L).\ $

$(x)\Rightarrow(i):\ $Take $L=Am$ for every $m\in M-P$. Then by part
(x),\ $(P:_{A}L)=(P:_{A}m)$ is a 1-absorbing prime ideal. The rest follows
from Theorem \ref{tmain} $(xi)\Rightarrow(i).$
\end{proof}

\begin{proposition}
Let $M\ $be an $A$-module an $P\ $be a classical 1-absorbing prime submodule
of $M.\ $Then for every nonunits $a,b,c\in A$ and $m\in M,\ (P:_{A}%
abcm)=(P:_{A}abm)\cup(P:_{A}cm).\ $In this case, either $(P:_{A}%
abcm)=(P:_{A}abm)$ or $(P:_{A}abcm)=(P:_{A}cm).$
\end{proposition}

\begin{proof}
Let $a,b,c\in A$ be nonunits and $m\in M.$ Choose $x\in(P:_{A}abcm).\ $Then we
have $abc(xm)\in P$ which implies that $abxm\in P$ or $cxm\in P.\ $Then we
conclude that $x\in(P:_{A}abm)\cup(P:_{A}cm).\ $The rest follows from the fact
that if an ideal is contained in a union of two ideals, then it must be
contained in one of them.
\end{proof}

\begin{theorem}
\label{tmult}Let $M$ be a multiplication $A$-module and $P$ a proper submodule
of $M.\ $The following statements are equivalent.

(i) $P\ $is a 1-absorbing prime submodule of $M.$

(ii)\ $P\ $is a classical 1-absorbing prime submodule of $M.$

(iii)\ $(P:_{A}M)\ $is a 1-absorbing prime ideal of $A.$
\end{theorem}

\begin{proof}
$(i)\Rightarrow(ii):\ $Follows from Proposition \ref{pro2}.

$(ii)\Rightarrow(iii):\ $Follows from Theorem \ref{tmain2}.

$(iii)\Rightarrow(ii):\ $Let $IJL\subseteq P$ for some proper ideals $I,J\ $of
$A\ $and some submodule $L$ of $M.\ $Since $M\ $is multiplication module, we
can write $L=KM$ for some ideal $K\ $of $A.$If $K=A,\ $then we have
$IJL=IJM\subseteq P$ which implies that $IJ\subseteq(P:_{A}M).$\ So assume
that $K$ is a proper ideal. Since $IJL=IJKM\subseteq P$, we have
$IJK\subseteq(P:_{A}M).\ $As $(P:_{A}M)\ $is a 1-absorbing prime ideal of
$A,\ $either $IJ\subseteq(P:_{A}M)\ $or $K\subseteq(P:_{A}M).\ $If
$IJ\subseteq(P:_{A}M),$ then we are done. So assume that $K\subseteq
(P:_{A}M).\ $This implies that $L=KM\subseteq P,$ so by \cite[Theorem
2.8]{X11}, $P\ $is a 1-absorbing prime submodule of $M.$
\end{proof}

\begin{definition}
\label{dmultiplicatively}Let $M$ be an $A$-module and $S$ be a subset of
$M-\{0\}$. If for all proper ideals $I,J,K$ of $M$ and all submodules $L,N$ of
$M$; $\left(  N+IJL\right)  \cap S\neq\varnothing$ and $\left(  N+KL\right)
\cap S\neq\varnothing$ implies $\left(  N+IJKL\right)  \cap S\neq\varnothing$,
then $S$ is said to be a classical 1-absorbing prime m-closed set.
\end{definition}

\begin{proposition}
\label{pro9}Let $M$ be an $A$-module and$\ P$ a proper submodule of $M$. Then
$P$ is a classical 1-absorbing prime submodule if and only if $S=M-P$ is a
classical 1-absorbing prime m-closed set.
\end{proposition}

\begin{proof}
Suppose that $P$ is classical 1-absorbing submodule of $M$ and $I,J,K$ are
proper ideals of $A$ and $N,L$ are submodules of $M$ such that $\left(
N+IJL\right)  \cap S\neq\varnothing$ and $\left(  N+KL\right)  \cap
S\neq\varnothing$ where $S=M-N$. Now, we will show that $\left(
N+IJKL\right)  \cap S\neq\varnothing.\ $Suppose to the contrary. Then we have
$N+IJKL\subseteq P$ which implies that $N\subseteq P$ and $IJKL\subseteq P$.
As $P\ $is a classical 1-absorbing prime submodule, by Theorem \ref{tmain2},
we have $IJL\subseteq P$ or $KL\subseteq P.\ $Then we conclude that $\left(
N+IJL\right)  \cap S=\emptyset$ or $\left(  N+KL\right)  \cap S=\emptyset$
which both of them are contradictions. Thus $S=M-P$ is a classical 1-absorbing
prime m-closed set. For the converse, assume that $S=M-P$ is a classical
1-absorbing prime m-closed set and $IJKL\subseteq P$ for some proper ideals
$I,J,K\ $of $A$ and some submodule $L$ of $M$. Choose $N=(0).\ $Then note that
$N+IJKL=IJKL\subseteq P,\ $that is, $(N+IJKL)\cap S=\emptyset.\ $Then we have
either $(N+IJL)\cap S=\emptyset$ or $(N+KL)\cap S=\emptyset\ $since $S=M-P$ is
a classical 1-absorbing prime m-closed set. This implies that $IJL\subseteq P$
or $KL\subseteq P.\ $Then by Theorem \ref{tmain2}, $P$ is a classical
1-absorbing prime submodule.
\end{proof}

\begin{theorem}
\label{tkrull}(\textbf{Krull seperation lemma for classical 1-absorbing prime
submodules) }Let $M$ be an $A$-module and $S$ be a classical 1-absorbing prime
m-closed set. Then the set of all submodules of $M$ which are disjoint from
$S$ has at least one maximal element. Any such maximal element is a classical
1-absorbing prime submodule.
\end{theorem}

\begin{proof}
Let $\Gamma=\left\{  K:K\text{ is submodule of }M\text{ and }K\cap
S=\emptyset\right\}  $. Note that $L=(0)\in\Gamma$ and $\Gamma$ is nonempty.
Since $(\Gamma,\subseteq)$ is a partially ordered set, by using Zorn's lemma,
we have at least one maximal element $P\ $of $\Gamma$. Now, we will show $P$
is classical 1-absorbing prime submodule. To prove this, take some proper
ideals $I,J,K$ of $A$ and some submodule $L$ of $M\ $such that $IJKL\subseteq
P.\ $Assume that $IJL\nsubseteq P$ and $KL\nsubseteq P$. By the maximality of
$P,$ we get $\left(  IJL+P\right)  \cap S\neq\emptyset$ and $\left(
KL+P\right)  \cap S\neq\emptyset$. Since $S$ is a classical 1-absorbing primem
m-closed set, we have $\left(  IJKL+P\right)  \cap S\neq\emptyset$. Hence
$P\cap S\neq\varnothing$, which is a contradiction. Therefore we have
$IJL\subseteq P\ $or $KL\subseteq P$. Thus, by Theorem \ref{tmain2},$P$ is a
classical 1-absorbing prime submodule of $M.$
\end{proof}

Recall from \cite{Mac} that an $A$-module $F$ is said to be a flat $A$-module
if for each exact sequence $K\rightarrow L\rightarrow M$ of $A$-modules, the
sequence $F\otimes K\rightarrow F\otimes L\rightarrow F\otimes M$ is also
exact. Also, $F\ $is said to be a faithfully flat the sequence $K\rightarrow
L\rightarrow M\ $is exact if and only if $F\otimes K\rightarrow F\otimes
L\rightarrow F\otimes M$ is exact. Azizi in \cite[Lemma 3.2]{X4} showed that
if $M\ $is an $A$-module, $P\ $is a submodule of $M\ $and $F\ $is a flat
$A$-module, then $(F\otimes P:_{F\otimes M}a)=F\otimes(P:_{M}a)$ for every
$a\in A.$

\begin{theorem}
\label{theo6} Let $M\ $be an $A$-module and $F\ $be a flat $A$-module. If $P$
is a classical 1-absorbing prime submodule of $M$ such that $F\otimes P\neq
F\otimes M$, then $F\otimes P$ is a classical 1-absorbing prime submodule of
$F\otimes M$.
\end{theorem}

\begin{proof}
Let $a,b,c\in A\ $be nonunits. Then by Theorem \ref{tmain2}, $\left(
P:_{M}abc\right)  =\left(  P:_{M}ab\right)  $ or $\left(  P:_{M}abc\right)
=\left(  P:_{M}c\right)  $. Assume that $\left(  P:_{M}abc\right)  =\left(
P:_{M}ab\right)  $. Then by \cite[Lemma 3.2]{X4} , $\left(  F\otimes
P:_{F\otimes M}abc\right)  =F\otimes\left(  P:_{M}abc\right)  =F\otimes\left(
P:_{M}ab\right)  =\left(  F\otimes P:_{F\otimes M}ab\right)  $. If $\left(
P:_{M}abc\right)  =\left(  P:_{M}c\right)  ,\ $then $\left(  F\otimes
P:_{F\otimes M}abc\right)  =F\otimes\left(  P:_{M}abc\right)  =F\otimes\left(
P:_{M}c\right)  =\left(  F\otimes P:_{F\otimes M}c\right)  .\ $Again by
Theorem \ref{tmain2}, $F\otimes P$ is a classical 1-absorbing prime submodule
of $F\otimes M$.
\end{proof}

\begin{theorem}
\label{ttensor}Let $M\ $be an $A$-module and $F\ $be a faithfully flat
$A$-module. Then $P$ is a classical 1-absorbing prime submodule of $M$ if and
only if $F\otimes P$ is a classical 1-absorbing prime submodule of $F\otimes
M$.
\end{theorem}

\begin{proof}
Let $P$ be a classical 1-absorbing prime submodule of $M$ and assume $F\otimes
P=F\otimes M$. Then $0\rightarrow F\otimes P\overset{\subseteq}{\rightarrow
}F\otimes M\rightarrow0$ is an exact sequence. Since $F$ is a faithfully flat
$A$-module, $0\rightarrow P\overset{\subseteq}{\rightarrow}M\rightarrow0$ is
an exact sequence. This implies that $P=M$, which is a contradiction. So
$F\otimes P\neq F\otimes M$. Then $F\otimes P$ is a classical 1-absorbing
prime submodule by Theorem \ref{theo6}. For the converse, let $F\otimes P$ is
a classical 1-absorbing prime submodule of $F\otimes M$. Then note that
$F\otimes P\neq F\otimes M$ and so $P\neq M$ . Let $a,b,c\in A$ be nonunits.
Then $\left(  F\otimes P:_{F\otimes M}abc\right)  =\left(  F\otimes
P:_{F\otimes M}ab\right)  $ or $\left(  F\otimes P:_{F\otimes M}abc\right)
=\left(  F\otimes P:_{F\otimes M}c\right)  $ by Theorem \ref{tmain2}. Assume
that $\left(  F\otimes P:_{F\otimes M}abc\right)  =\left(  F\otimes
P:_{F\otimes M}ab\right)  $. Hence $F\otimes\left(  P:_{M}ab\right)  =\left(
F\otimes P:_{F\otimes M}ab\right)  =\left(  F\otimes P:_{F\otimes
M}abc\right)  =F\otimes\left(  P:_{M}abc\right)  $. This implies that
$0\rightarrow F\otimes\left(  P:_{M}ab\right)  \overset{\subseteq}%
{\rightarrow}F\otimes\left(  P:_{M}abc\right)  \rightarrow0$ is exact
sequence. Since $F$ is a faithfully flat $A$-module $0\rightarrow\left(
P:_{M}ab\right)  \overset{\subseteq}{\rightarrow}\left(  P:_{M}abc\right)
\rightarrow0$ is an exact sequence which implies that $\left(  P:_{M}%
ab\right)  =\left(  P:_{M}abc\right)  $. In other case, one can similarly show
that $\left(  P:_{M}c\right)  =\left(  P:_{M}abc\right)  .\ $Consequently, $P$
is a classical 1-absorbing prime submodule of $M$ by Theorem \ref{tmain2}.
\end{proof}

\begin{corollary}
Let $M\ $be an $A$-module and $X$ be an indeterminate. If $P$ is a classical
1-absorbing prime submodule of $M$, then $P\left[  X\right]  $ is a classical
1-absorbing prime submodule of $M\left[  X\right]  $.
\end{corollary}

\begin{proof}
Assume that $P$ is a classical 1-absorbing prime submodule of $M$. Note that
$A\left[  X\right]  $ is a flat $A$-module. So by Theorem \ref{theo6},
$A\left[  X\right]  \otimes P$ is a classical 1-absorbing prime submodule of
$A\left[  X\right]  \otimes M$. Since $A\left[  X\right]  \otimes P\simeq
P\left[  X\right]  $ and $A\left[  X\right]  \otimes M\simeq M\left[
X\right]  ,\ P\left[  X\right]  $ is a classical 1-absorbing prime submodule
of $M\left[  X\right]  $.
\end{proof}

\begin{proposition}
\label{quotient}Let $M$\ be an $A$-module, $P$ be a classical 1-absorbing
prime submodule of $M\ $and $S$ be a multiplicatively closed subset of $A$
such that $\left(  P:_{A}M\right)  \cap S=\emptyset.\ $Then $S^{-1}P$ is a
classical 1-absorbing prime submodule of $S^{-1}M$.
\end{proposition}

\begin{proof}
Let $P$ be a classical 1-absorbing prime submodule of $M$ and $\left(
P:_{A}M\right)  \cap S=\emptyset$. Suppose that $\frac{a}{s}\frac{b}{t}%
\frac{c}{u}\frac{m}{v}\in S^{-1}P$ for some nonunits $\frac{a}{s},\frac{b}%
{t},\frac{c}{u}\in S^{-1}A$ and $\frac{m}{v}\in S^{-1}M$. Then there exists
$z\in S$ such that $z(abcm)=abc(zm)\in P$. Since $P$ is a classical
1-absorbing prime submodule of $M,$ we have $ab(zm)\in P$ or $c(zm)\in P$.
This implies that $\frac{abm}{stv}=\frac{abzm}{stzv}\in S^{-1}P$ or $\frac
{cm}{uv}=\frac{czm}{zuv}\in S^{-1}P$. Consequently, $S^{-1}P$ is a classical
1-absorbing prime submodule of $S^{-1}M$.
\end{proof}

Let $A_{i}$ be a commutative ring with identity and $M_{i}$ be an $A_{i}%
$-module for $i=1,2.$ Let $A=A_{1}\times A_{2}$ and $M=M_{1}\times M_{2}%
.\ $Then $M\ $is an $A$-module and each submodule $P\ $of $M\ $has the form
$P=P_{1}\times P_{2}$ for some submodules $P_{1}$ of $M_{1}$ and $P_{2}$ of
$M_{2}$.

\begin{theorem}
\label{tcar}Let $A=A_{1}\times A_{2}\ $be a decomposable ring and
$M=M_{1}\times M_{2}$ be an $A$-module where $M_{1}$ is an $A_{1}$-module and
$M_{2}$ is an $A_{2}$-module. Suppose that $P=P_{1}\times P_{2}$ is a proper
submodule of $M$. The following statements are equivalent.

(i) $P\ $is a classical 1-absorbing prime submodule of $M.$

(ii)\ $P_{1}=M_{1}\ $and $P_{2}\ $is a classical prime submodule of $M_{2}%
\ $or $P_{2}=M_{2}\ $and $P_{1}\ $is a classical prime submodule of $M_{1}$.

(iii) $P\ $is a classical prime submodule of $M.$
\end{theorem}

\begin{proof}
$(i)\Rightarrow(ii):\ $Suppose that $P\ $is a classical 1-absorbing prime
submodule of $M.$\ Then by Theorem \ref{tmain2}, $(P:_{A}M)=(P_{1}:_{A_{1}%
}M_{1})\times(P_{2}:_{A_{2}}M_{2})$ is a 1-absorbing prime ideal of $A.\ $Then
by \cite[Corollary 2.5]{X8}, $(P_{1}:_{A_{1}}M_{1})=A_{1}\ $or $(P_{2}%
:_{A_{2}}M_{2})=A_{2}.\ $Then we conclude that either $P_{1}=M_{1}$ or
$P_{2}=M_{2}.\ $Without loss of generality we may assume that $P_{2}=M_{2}%
.\ $Let $abm\in P_{1}\ $for some $a,b\in A_{1}$ and $m\in M_{1}.$\ If $a$ or
$b$ is unit, then $am\in P_{1}\ $or $bm\in P_{1}.\ $Now, assume that
$a,b\ $are nonunits in $A_{1}.\ $Then we have $(a,0)(1,0)(b,0)(m,0)=(abm,0)\in
P.\ $Since $P\ $is a classical 1-absorbing prime submodule of $M,\ $we
$(a,0)(1,0)(m,0)=(am,0)\in P$ or $(b,0)(m,0)=(bm,0)\in P.\ $This implies that
$am\in P_{1}\ $or $bm\in P_{1}.\ $Thus, $P_{1}\ $is a classical 1-absorbing
prime submodule of $M_{1}.\ $One can similarly show that $P_{2}\ $is a
classical 1-absorbing prime submodule of $M_{2}.$

$(ii)\Rightarrow(i):\ $Without loss of generality we may assume that
$P_{1}=M_{1}\ $and $P_{2}\ $is a classical prime submodule of $M_{2}.\ $Let
$(a,b)(c,d)(m,n)=(acm,bdn)\in P$ as $(a,b),(c,d)\in A$ and $(m,n)\in M.\ $Then
we have $bdn\in P_{2}.\ $As $P_{2}\ $is a classical prime submodule of
$M_{2},\ $we conclude that either $cn\in P_{2}\ $or $dn\in P_{2}.\ $This
implies that $(a,b)(m,n)\in P$ or $(c,d)(m,n)\in P$. Thus, $P\ $is a classical
prime submodule of $M.$

$(iii)\Rightarrow(i):\ $Follows from Proposition \ref{p1}.
\end{proof}

\begin{lemma}
\label{lem1} Let $A=A_{1}\times A_{2}\times\cdots\times A_{n}$ be a
decomposable ring and $M=M_{1}\times M_{2}\times\cdots\times M_{n}$ be an
$A$-module where for every $1\leq i\leq n$, $M_{i}$ is an $A_{i}$-module,
respectively. A proper submodule $P$ of $M$ is a classical prime submodule of
$M$ if and only if $P=%
%TCIMACRO{\tprod \limits_{i=1}^{n}}%
%BeginExpansion
{\textstyle\prod\limits_{i=1}^{n}}
%EndExpansion
P_{i}$ such that for some $k\in\left\{  1,2,..,n\right\}  $, $P_{k}$ is a
classical prime submodule of $M_{k}$, and $P_{i}=M_{i}$ for every
$i\in\left\{  1,2,..,n\right\}  \diagdown\left\{  k\right\}  $.
\end{lemma}

\begin{proof}
Follows from \cite[Lemma 1]{X18}.
\end{proof}

\begin{theorem}
\label{tcargen}Let $A=A_{1}\times A_{2}\times\cdots\times A_{n}\ $be a
decomposable ring and $M=M_{1}\times M_{2}\times\cdots\times M_{n}$ be an
$A$-module where for every $1\leq i\leq n$, $M_{i}$ is an $A_{i}$-module,
respectively. For a proper submodule $P$ of $M,$ the following statements are equivalent.

(i)$\ P$ is a classical 1-absorbing prime submodule of $M$.

(ii)\ $P=%
%TCIMACRO{\tprod \limits_{i=1}^{n}}%
%BeginExpansion
{\textstyle\prod\limits_{i=1}^{n}}
%EndExpansion
P_{i}$ such that for some $k\in\left\{  1,2,..,n\right\}  $ , $P_{k}$ is a
classical 1-absorbing prime submodule of $M_{k}$ and $P_{t}=M_{t}$ for every
$t\in\left\{  1,2,..,n\right\}  \diagdown\left\{  k\right\}  $.

(iii)\ $P\ $is a classical prime submodule of $M.$
\end{theorem}

\begin{proof}
$(i)\Rightarrow(ii):\ $Follows from induction on $n$ by using Theorem
\ref{tcar}.

$(ii)\Rightarrow(iii):\ $Follows from Lemma \ref{lem1}.

$(iii)\Rightarrow(i):\ $Follows from Proposition \ref{p1}.
\end{proof}

\begin{theorem}
\label{theoremfin}Let $M$ be an $A$-module and $P\ $be a classical 1-absorbing
prime submodule which is not classical prime. Then $A\ $is a local ring with
unique maximal ideal $\mathfrak{q}$ of $A$ such that $\mathfrak{q}%
^{2}\subseteq(P:_{A}m)$ for some $m\in M-P.\ $
\end{theorem}

\begin{proof}
Suppose that $P$ is a classical 1-absorbing prime submodule of $M$\ that is
not a classical prime submodule. Since $P$ is not classical prime, there
exists $m\in M-P$ such that $(P:_{A}m)$ is not a prime ideal of $A$. On the
other hand, since $P\ $is a classical 1-absorbing prime submodule of
$M,\ $then by Theorem \ref{tmain}, $(P:_{A}m)$ is a 1-absorbing prime ideal of
$A.\ $Thus, $A\ $admits a 1-absorbing prime ideal that is not prime. Then by
\cite[Lemma 2.1]{Abdel}, $A\ $is a local ring with unique maximal ideal
$\mathfrak{q}$ of $A$ such that $\mathfrak{q}^{2}\subseteq(P:_{A}m)$ for some
$m\in M-P.$
\end{proof}

\section{Classical prime like conditions in amalgamated duplication of a
module}

Let\ $A$ be a ring and $I$ be an ideal of $A$. Recall from \cite{Danna} that
the amalgamated duplication of a ring $A$ along an ideal $I,$ denoted by
$A\bowtie I=\{(a,a+i):a\in A,i\in I\},$ is a special subring of $A\times A$
with componentwise addition and multiplication. In fact, $A\bowtie I$ is a
commutative subring having the same identity of $A\times A.\ $Let $M\ $be an
$A$-module and $I$ be an ideal of $A.\ $Recall from \cite{Bou} that the
amalgamated duplication of an $A$-module $M$ along an ideal $I$, denoted by
$M\bowtie I=\{(m,m+m^{\prime}):m\in M,\ m^{\prime}\in IM\}$ is an $A\bowtie
I$- module with componentwise addition and the following scalar
multiplication:\ $(a,a+i)(m,m+m^{\prime})=(am,(a+i)(m+m^{\prime}))$ for each
$(a,a+i)\in A\bowtie I$ and $(m,m+m^{\prime})\in M\bowtie I\ $\cite{Bou}. Note
that if we consider $M=A$ as an $A$-module, then $M\bowtie I$ and $A\bowtie I$
coincide. In this part of the paper, we investigate the classical prime
submodules, classical 1-absorbing prime submodules and classical 2-absorbing
submodules of the amalgamated duplication $M\bowtie I$ of an $A$-module
$M\ $along an ideal $I$. We start with the following lemmas which will be
frequently used in our main theorem (Theorem \ref{tmainnn}).

\begin{lemma}
\label{lemfin}Let $M$ be an $A$-module, $P\ $be a submodule of $M$ and $I\ $be
an ideal of $A.\ $Consider the $A\bowtie I$-module $M\bowtie I.\ $The
following statements are satisfied:

(i) $(P\bowtie I:_{M\bowtie I}(a,a+i))=(P:_{M}a)\bowtie I$ for every $a\in A$
and $i\in I.$

(ii) $(P\bowtie I:_{A\bowtie I}(m,m+m^{\prime}))=(P:_{A}m)\bowtie I$ for every
$m\in M$ and $m^{\prime}\in IM.$
\end{lemma}

\begin{proof}
It is easy to see that if $P\ $is a submodule of $M,\ $then $P\bowtie
I=\{(m,m+m^{\prime}):m\in P,m^{\prime}\in IM\}$ is an $A\bowtie I$-submodule
of $M\bowtie I.$

(i):\ Let $(m,m+m^{\prime})\in(P:_{M}a)\bowtie I$ for some $m\in M,m^{\prime
}\in IM.\ $Then we have $m\in(P:_{M}a)$ which implies that $am\in P.\ $Then we
conclude that $(a,a+i)(m,m+m^{\prime})=(am,am+am^{\prime}+im+im^{\prime})\in
P\bowtie I$ as $am^{\prime}+im+im^{\prime}\in IM.\ $Thus we have
$(m,m+m^{\prime})\in(P\bowtie I:_{M\bowtie I}(a,a+i)).\ $For the converse,
take $(m,m+m^{\prime})\in(P\bowtie I:_{M\bowtie I}(a,a+i)).\ $Then we have
$(a,a+i)(m,m+m^{\prime})\in P\bowtie I$ which implies that $am\in P$, that is,
$m\in(P:_{M}a).\ $This gives$\ (m,m+m^{\prime})\in(P:_{M}a)\bowtie I\ $which
completes the proof.

(ii):\ It is similar to (i).
\end{proof}

\begin{lemma}
\label{lemfin2}(\cite[Lemma 3.1]{X4}) Let $M\ $be an $A$-module and $P\ $a
submodule of $M.\ $The following $P$ is a classical prime submodule if and
only if $(P:_{M}ab)=(P:_{M}a)$ or $(P:_{M}ab)=(P:_{M}b)$ for every $a,b\in
A.\ $
\end{lemma}

\begin{lemma}
\label{lemfin3}Let $A$ be a ring, $I$ be an ideal of $A\ $and consider the
amalgamated duplication $A\bowtie I.\ $For every $i\in I$ and $a\in
A,\ (a,a+i)$ is a unit of $A\bowtie I$ if and only if $a\ $and $a+i$ are units
in $A.$
\end{lemma}

\begin{proof}
The "if part" is clear. For the "only if part", assume that $a\ $and $a+i$ are
units in $A.$\ Then there exists $x,y\in A$ such that $ax=1=(a+i)y$.\ Then we
conclude that $(a+i)(x-yxi)=ax-yaxi+ix-yxi^{2}=ax+ix-ixy(a+i)=1$. Then we have
$(a,a+i)(x,x-yxi)=(1,1)$ which completes the proof.
\end{proof}

\begin{theorem}
\label{tmainnn}Let $M$ be an $A$-module, $P\ $be a submodule of $M$ and
$I\ $be an ideal of $A.\ $Consider the $A\bowtie I$-module $M\bowtie I.\ $The
following statements are satisfied:

(i)\ $P\bowtie I\ $is a classical prime submodule of $M\bowtie I$ if and only
if $P$ is a classical prime submodule of $M.\ $

(ii)\ $P\bowtie I\ $is a classical 1-absorbing prime submodule of $M\bowtie I$
if and only if $P$ is a classical 1-absorbing prime submodule of $M.\ $

(iii)\ $P\bowtie I\ $is a classical 2-absorbing submodule of $M\bowtie I$ if
and only if $P$ is a classical 2-absorbing submodule of $M.\ $
\end{theorem}

\begin{proof}
$(i):\ $Suppose that $P\bowtie I\ $is a classical prime submodule of $M\bowtie
I.\ $Let $abm\in P$ for some $a,b\in A$ and $m\in M.\ $Then we have
$(a,a)(b,b)(m,m)=(abm,abm)\in P\bowtie I.\ $As $P\bowtie I\ $is a classical
prime submodule of $M\bowtie I,\ $we have $(a,a)(m,m)\in P\bowtie I$ or
$(b,b)(m,m)\in P\bowtie I$ which implies that $am\in P$ or $bm\in P.\ $Thus,
$P\ $is a classical prime submodule of $M.\ $For the converse, assume that $P$
is a classical prime submodule of $M.\ $Let $(a,a+i),(b,b+j)\in A\bowtie I$.
Then by Lemma \ref{lemfin}, $\left(  P\bowtie I:_{M\bowtie I}%
(a,a+i)(b,b+j)\right)  =(P:_{M}ab)\bowtie I.\ $Since $P\ $is a classical prime
submodule, by Lemma \ref{lemfin2}, we have $(P:_{M}ab)=(P:_{M}a)$ or
$(P:_{M}ab)=(P:_{M}b).\ $Without loss of generality, we may assume that
$(P:_{M}ab)=(P:_{M}a)$. Again by Lemma \ref{lemfin}, we conclude that
\begin{align*}
\left(  P\bowtie I:_{M\bowtie I}(a,a+i)(b,b+j)\right)   &  =(P:_{M}ab)\bowtie
I\\
&  =(P:_{M}a)\bowtie I\\
&  =\left(  P\bowtie I:_{M\bowtie I}(a,a+i)\right)  .
\end{align*}
In other case, one similarly have $\left(  P\bowtie I:_{M\bowtie
I}(a,a+i)(b,b+j)\right)  =\left(  P\bowtie I:_{M\bowtie I}(b,b+j)\right)  $.
Then by Lemma \ref{lemfin2}, $P\bowtie I\ $is a classical prime submodule of
$M\bowtie I$.

$(ii):\ $Assume that $P\bowtie I\ $is a classical 1-absorbing prime submodule
of $M\bowtie I$. Let $abcm\in P$ for some nonunits $a,b,c\in A$ and $m\in
M.\ $Then by Lemma \ref{lemfin3}, $(a,a),(b,b),(c,c)\in A\bowtie I$ are
nonunits and $(a,a)(b,b)(c,c)(m,m)\in P\bowtie I.\ $Since $P\bowtie I\ $is a
classical 1-absorbing prime submodule, we have $(a,a)(b,b)(m,m)\in P\bowtie I$
or $(c,c)(m,m)\in P\bowtie I.\ $This implies that $abm\in P$ or $cm\in P$,
that is, $P$ is a classical 1-absorbing prime submodule of $M.\ $For the
converse, assume that $P$ is a classical 1-absorbing prime submodule of
$M.\ $If $P\ $is a classical prime submodule of $M,\ $then by (i),\ $P\bowtie
I$ is a classical prime submodule of $M\bowtie I.\ $Then by Proposition
\ref{p1}, $P\bowtie I$ is a classical 1-absorbing prime submodule of $M\bowtie
I$.\ Now, assume that $P\ $is not a classical prime submodule of $M.\ $Then by
Theorem \ref{theoremfin}, $(A,\mathfrak{q})$ is a local ring such that
$\mathfrak{q}^{2}\subseteq(P:_{A}m)$ for some $m\in M-P.\ $Since
$I\subseteq\mathfrak{q}=Jac(A),\ $where $Jac(A)$ is the Jacobson radical of
$A,\ $by Lemma \ref{lemfin3}, $(a,a+i)$ is a unit in $A\bowtie I$ if and only
if $a$ is a unit in $A.\ $Now, let $(a,a+i),(b,b+j),(c,c+k)$ be nonunits in
$A\bowtie I.\ $Then we have $a,b,c$ are nonunits in $A.\ $Since $P\ $is a
classical 1-absorbing prime submodule of $M,\ $by Theorem \ref{tmain2},
$(P:_{M}abc)=(P:_{M}ab)$ or $(P:_{M}abc)=(P:_{M}c).\ $Without loss of
generality, we may assume that $(P:_{M}abc)=(P:_{M}ab).\ $Then by Lemma
\ref{lemfin},
\begin{align*}
\left(  P\bowtie I:_{M\bowtie I}(a,a+i)(b,b+j)(c,c+k\right)  )  &
=(P:_{M}abc)\bowtie I\\
&  =(P:_{M}ab)\bowtie I\\
&  =(P\bowtie I:_{M\bowtie I}(a,a+i)(b,b+j)).
\end{align*}
In other case, one can similarly have,
\[
\left(  P\bowtie I:_{M\bowtie I}(a,a+i)(b,b+j)(c,c+k\right)  )=(P\bowtie
I:_{M\bowtie I}(c,c+k)).
\]
Then by Theorem \ref{tmain2}, $P\bowtie I\ $is a classical 1-absorbing prime
submodule of $M\bowtie I.$

(iii):\ Let $P\bowtie I\ $be a classical 2-absorbing submodule of $M\bowtie
I.\ $A similar argument in (i) or (ii) shows that $P\ $is a classical
2-absorbing submodule of $M.\ $Now, for the converse assume that $P\ $is a
classical 2-absorbing submodule of $M.$ Choose, $(a,a+i),(b,b+j)$ and
$(c,c+k)\in A\bowtie I.\ $Then by Lemma \ref{lemfin}, we have
\[
\left(  P\bowtie I:_{M\bowtie I}(a,a+i)(b,b+j)(c,c+k\right)  )=(P:_{M}%
abc)\bowtie I.
\]
Also, by \cite[Lemma 2]{X18}, we have $(P:_{M}abc)=(P:_{M}ab)\cup
(P:_{M}ac)\cup(P:_{M}bc).\ $Now, let $(m,m+m^{\prime})\in\left(  P\bowtie
I:_{M\bowtie I}(a,a+i)(b,b+j)(c,c+k\right)  ).\ $Then by \ref{lemfin}, we have
$(m,m+m^{\prime})\in(P:_{M}abc)\bowtie I\ $which implies that $m\in
(P:_{M}abc)=(P:_{M}ab)\cup(P:_{M}ac)\cup(P:_{M}bc).\ $Without loss of
generality, we may assume that $m\in(P:_{M}ab).\ $This gives $(m,m+m^{\prime
})\in(P:_{M}ab)\bowtie I=(P\bowtie I:_{M\bowtie I}(a,a+i)(b,b+j)).\ $Then we
have $\left(  P\bowtie I:_{M\bowtie I}(a,a+i)(b,b+j)(c,c+k\right)  )$ is a
subset of $(P\bowtie I:_{M\bowtie I}(a,a+i)(b,b+j))\cup(P\bowtie I:_{M\bowtie
I}(a,a+i)(c,c+k))\cup(P\bowtie I:_{M\bowtie I}(b,b+j)(c,c+k))$. Since the
other containment is always true, we have the equality
\begin{align*}
&  \left(  P\bowtie I:_{M\bowtie I}(a,a+i)(b,b+j)(c,c+k\right)  )\\
&  =(P\bowtie I:_{M\bowtie I}(a,a+i)(b,b+j))\cup(P\bowtie I:_{M\bowtie
I}(a,a+i)(c,c+k))\cup(P\bowtie I:_{M\bowtie I}(b,b+j)(c,c+k)).
\end{align*}
Then by by \cite[Lemma 2]{X18}, $P\bowtie I\ $is a classical 2-absorbing
submodule of $M\bowtie I$.
\end{proof}

\end{document}